\documentclass{article}

\usepackage{amsmath,amssymb,amsthm,relsize}

\newcommand{\Z}{\mathbb{Z}}

\newtheorem{thm}{Theorem}
\newtheorem{lemma}[thm]{Lemma}

\begin{document}

\title{Diophantine approximation and
the equation $(a^2cx^k - 1)(b^2cy^k - 1) = (abcz^k - 1)^2$}

\author{Eva G. Goedhart and Helen G. Grundman}

\date{}

\maketitle

\begin{abstract}
We prove that the Diophantine equation 
$(a^2cx^k - 1)(b^2cy^k - 1) = (abcz^k - 1)^2$
has no solutions in positive integers with 
$x$, $y$, $z > 1$, $k \geq 7$ and
$a^2x^k \neq b^2y^k$.  
\end{abstract}

\section{Introduction}
\label{intro}

In 2003, Bugeaud and Dujella~\cite{BuDu03} asked when there exists
a set of $m$ positive integers such that one more than the product 
of any two of the integers is a $k$th power, with $k \geq 3$.
Later, Bugeaud~\cite{Bu04} considered the case in which the 
set is of the form $\{1,A,B\}$, with $1 < A < B$.  
He noted that the desired property is equivalent to 
the existence of an integer solution to the equation
$(x^k - 1)(y^k - 1) = (z^k - 1)$ (with appropriate restrictions),
and used this to prove
that this is only possible if $k \leq 74$.  
He also considered variations of the problem, one of which
leads to the equation 
$(x^k - 1)(y^k - 1) = (z^k - 1)^2$.
Bennett~\cite{Be07} proved that each of these equations has
no integer solutions (again, with appropriate restrictions)
if $k\geq 4$, and that this bound on $k$ is sharp.

Of interest here is a 2014 generalization of Bennett's result
by Zhang~\cite{Zh14}, in which it is shown that the equation
$(ax^k - 1)(by^k - 1) = (abz^k - 1)$ has no solutions with 
$a$, $b \in \Z^+$, $|x| > 1$, $|y| > 1$, and $k \geq 4$.
In this paper, we consider a similar modification of the equation
$(x^k - 1)(y^k - 1) = (z^k - 1)^2$
and prove that it has no solutions with $k \geq 7$.

\begin{thm}
\label{mainthm}
Let $a$, $b$, $c$, and $k$ be positive integers with $k \geq 7$.  
The equation
\begin{equation}
\label{maineq}
(a^2cx^k - 1)(b^2cy^k - 1) = (abcz^k - 1)^2
\end{equation}
has no solution in integers with $x$, $y$, $z > 1$ 
and $a^2x^k \neq b^2y^k$.
\end{thm}

We prove Theorem~\ref{mainthm} in the following section.
Our main tools are continued fractions and the following
lemma from~\cite{Be97}, providing a bound on how well one can approximate
certain algebraic numbers by rational numbers.
For $n \geq 2$, define
\[\mu_n = \prod_{\substack{p\mid n\\
p \text{\ \it prime}}}
p^{1/(p-1)}.\]

\begin{lemma}[Bennett]
\label{bennett}

If $n$ and $N$ are positive integers with $n\geq 3$
and
\[\left(\sqrt{N} + \sqrt{N + 1}\right)^{2(n - 2)} >
\left(n\mu_n\right)^n,\]
then, for any $p$, $q\in \Z^+$,
\[\left|\sqrt[\mathlarger{n}]{1 + \frac{1}{N}} - \frac{p}{q} \right|
> (8n\mu_nN)^{-1}q^{-\lambda}\]
with
\[\lambda = 1 + 
\frac{\log\left(\left(\sqrt{N} + \sqrt{N + 1}\right)^2n\mu_n\right)}
{\log{\left(\left(\sqrt{N} + \sqrt{N + 1}\right)^2/n\mu_n\right)}}.\]
\end{lemma}

To summarize our proof,
we first assume that a solution,
$(k,a,b,c,x,y,z)$, to equation~(\ref{maineq}) exists, and prove 
in Lemma~\ref{notS} that it
yields a sufficiently good rational approximation of
\[\sqrt[\mathlarger{k}]{1 + \frac{1}{a^2cx^k - 1}}\]
so that, for large values of $k$, $a$, $c$, or $x$,
Lemma~\ref{bennett} yields a contradiction.
For the finite number of remaining cases, we prove that
one of the continued fraction convergents of
\[\frac{1}{x}\sqrt[\mathlarger{k}]{1 + \frac{1}{a^2cx^k - 1}}
= \sqrt[\mathlarger{k}]{\frac{a^2c}{a^2cx^k - 1}}\]
must equal $y/z^2$,
and derive a bound on the index of the convergent.
Computer calculations then yield a contradiction in each of these cases.

\section{Proof of Theorem~\ref{mainthm}}

Let $a$, $b$, $c$, and $k$, as in Theorem~\ref{mainthm}, be given,
and suppose that $x$, $y$, $z\in \Z^+$ satisfy equation~(\ref{maineq})
with $x$, $y$, $x > 1$ and $a^2x^k \neq b^2y^k$.
Since $(a^2cx^k - 1)(b^2cy^k - 1)$ is a square, there exist positive
integers $u$, $v$, and $w$ such that 
\begin{equation}\label{uvw}
a^2cx^k - 1 = uv^2 \hskip 15pt \text{and}\hskip 15pt  b^2cy^k - 1 = uw^2.
\end{equation}
Then, by equation~(\ref{maineq}), $uvw = abcz^k - 1$.
We assume, without loss of generality (since $a^2x^k \neq b^2y^k$), 
that $v < w$.

Since $w > v$,
$v^2 + w^2 = (w - v)^2 + 2vw > 2vw$ and so
$(a^2cx^k)(b^2cy^k) = (uv^2 + 1)(uw^2 + 1) =
(uvw)^2 + u(v^2 + w^2) + 1 > (uvw)^2 + 2uvw + 1 = (uvw + 1)^2 =
(abcz^k)^2$.  Hence $xy > z^2$.

We now verify that the hypothesis of Lemma~\ref{bennett}
holds with $N = a^2cx^k - 1 = uv^2$.
Note that $uv^2 = a^2cx^k - 1 \geq 2^k - 1 \geq 127$.  
For $t\in \Z^+$, since
$\sqrt{1 - 1/t}$ increases as $t$ increases, we have
$\sqrt{1 - 1/(uv^2 + 1)} \geq \sqrt{1 - 1/128} > .99$.
Thus, 
\begin{align*}
\left(\sqrt{uv^2} + \sqrt{uv^2 + 1}\right)^{2(k - 2)} & =
\left(\left(\sqrt{1 - \frac{1}{uv^2 + 1}} + 1\right)
\sqrt{uv^2 + 1}\right)^{2(k - 2)} \\
& > \left((.99 + 1)\left(2^{k/2}\right)\right)^{2(k - 2)}.
\end{align*}
Since $1.99^{1.01} > 2$ and $k \geq 7$ implies that $2^{k-.6} > k^2$,
\[\left((.99 + 1)\left(2^{k/2}\right)\right)^{2(k - 2)} >
2^{2(k - 2)/1.01 + k(k - 2)} > \left(2^{k - .6}\right)^k
> \left(k^2\right)^k > \left(k\mu_k\right)^k. \]
Therefore, we can apply Lemma~\ref{bennett} with
$N = uv^2$.  Letting
\[\alpha = \sqrt[\mathlarger{k}]{1 + \frac{1}{uv^2}}
\hskip 15pt \text{and}\hskip 15pt 
\beta = \frac{xy}{z^2}\] 
we find that
\begin{equation}\label{alpha-beta>}
\left|\alpha - \beta\right| =
\left| \sqrt[\mathlarger{k}]{1 + \frac{1}{uv^2}} - \frac{xy}{z^2}\right| >
(8k\mu_kuv^2)^{-1}z^{-2\lambda},
\end{equation}
with 
\begin{equation}\label{ourlambda}
\lambda = 1 +
\frac{\log\left(\left(\sqrt{uv^2} + \sqrt{uv^2 + 1}\right)^2k\mu_k\right)}
{\log{\left(\left(\sqrt{uv^2} + \sqrt{uv^2 + 1}\right)^2/k\mu_k\right)}}.
\end{equation} 

Note that
$\lambda = \Lambda_k(a^2cx^k)$, where, for $K\geq 7$ and $D \geq 2^k$,
\[\Lambda_K(D) = 2 +
\frac{2\log(K\mu_K)}{2\log(\sqrt{D - 1} + \sqrt{D})  -
\log(K\mu_K) }.\]
Since $x \geq 2$ and, 
for each value of $K$, $\Lambda_K$ is a decreasing function,
we have that 
\begin{equation}\label{2^k}
\lambda = \Lambda_k(a^2cx^k) \leq \Lambda_k(2^k).
\end{equation}

Now, for $K\geq 7$, let
\[\Lambda(K) = 2 + \frac{6\log K}{2(K + 1)\log 2 - 3\log K}.\]
Since $k \geq 7$, $\mu_k \leq \sqrt{k}$ and so
\begin{equation}\label{Lambdak}
\lambda \leq \Lambda_k(2^k) \leq
2 + \frac{2\log(k^{3/2})}
{2\log\left(2^{(k-1)/2} + 2^{k/2}\right) - \log{k^{3/2}}}
< \Lambda(k).
\end{equation}

Next, we compute an upper bound for $w^{k-2\lambda}$ 
(and thus for $w$), as follows.
Writing $\alpha$ and $\beta$ in terms of $u$, $v$, and $w$,
we have
\begin{equation*}
\begin{aligned}
\alpha^k - \beta^k & = \left(1 + \frac{1}{uv^2}\right)-
\frac{(uv^2 + 1)(uw^2 + 1)}{(uvw + 1)^2} \\
& = \frac{u v^2(2u v w - u v^2) + (2u v w + 1)}{u v^2(u v w + 1)^2},
\end{aligned}
\end{equation*}
which is clearly positive, since $w > v$.  Hence, $\alpha > \beta$.
Further, from the above it follows that 
\begin{equation}\label{alpk-betk}
\begin{aligned}
\alpha^k - \beta^k & < 
\frac{u v^2(2u v w + 2) + (2u v w + 2)}{u v^2(u v w + 1)^2} \\
& = \frac{2}{u v w + 1}\left(1 + \frac{1}{uv^2}\right)
= \frac{2\alpha^k}{uvw + 1}.
\end{aligned}
\end{equation}


Since $xy > z^2$, we have $\alpha > \beta > 1$, and so
\[\alpha^k - \beta^k = 
(\alpha - \beta)\sum_{i=0}^{k-1} \alpha^{k-1-i}\beta^i
> k(\alpha - \beta).\]
Combining this with inequalities~(\ref{alpha-beta>}) 
and~(\ref{alpk-betk}), we have
\begin{equation}\label{wub3}
(8k\mu_kuv^2)^{-1}z^{-2\lambda} < \alpha - \beta
< \frac{2\alpha^k}{k(uvw + 1)} < \frac{2\alpha^k}{kuvw}.
\end{equation}
We also have that \[z^k = \frac{uvw + 1}{abc}
\leq \left(1 + \frac{1}{uvw}\right)uvw < \alpha^kuvw.\]
Using this to eliminate $z$ in
inequality~(\ref{wub3}), then 
solving for $w^{k-2\lambda}$ yields 
\begin{equation}\label{wub}
w^{k-2\lambda} < 2^{4k}\mu_k^k \alpha^{k(k+2\lambda)}
u^{2\lambda}v^{k+2\lambda}.
\end{equation}

To compute a lower bound for $w^2$ (and thus for $w$) 
we note that $xy > z^2$ implies that $xy \geq z^2 + 1$.
Using this with~(\ref{uvw}) and
the binomial theorem, we deduce that
\begin{equation}
\label{wlb1}
\begin{aligned}
uv^2 + uw^2 & = (uv^2 + 1)(uw^2 + 1) - (uvw + 1)^2 + 2uvw \\
& > (a^2cx^k)(b^2cy^k) - (abcz^k)^2 \\
& \geq a^2b^2c^2(z^2 + 1)^k - a^2b^2c^2z^{2k} \\
& > a^2b^2c^2kz^{2(k-1)} + a^2b^2c^2\frac{k(k-1)}{2}z^{2(k-2)}.
\end{aligned}
\end{equation}

Now, $k\geq 7$ and $a$, $b$, $c\geq 1$ imply that
$a^2b^2c^2\frac{k(k-1)}{2}z^{2(k-2)} > abcz^k = uvw + 1 > uv^2$.
Thus, inequality~(\ref{wlb1}) implies that 
$uw^2 > a^2b^2c^2kz^{2(k-1)}$.
So
\[(uw^{2})^k > (a^{2}b^{2}c^{2}kz^{2(k-1)})^k 
\geq k^k(abcz^k)^{2(k-1)}
> k^k(uvw)^{2(k-1)}.\]
Thus, 
\begin{equation}\label{wlb}
w^2 > k^{k}u^{(k-2)}v^{2(k-1)}.
\end{equation}

Combining these two bounds on $w$,
inequalities~(\ref{wub}) and~(\ref{wlb}),
we eliminate $w$ to obtain
\[k^{k(k-2\lambda)}u^{(k-2)(k-2\lambda)}v^{2(k-1)(k-2\lambda)}
< 2^{8k}\mu_k^{2k} \alpha^{2k(k+2\lambda)}
u^{4\lambda}v^{2(k+2\lambda)},\]
and so
\begin{equation}\label{uv^2}
\left(uv^2\right)^{(k-2\lambda-2)} < 
2^{8}\mu_k^{2} \alpha^{2(k+2\lambda)}k^{-(k-2\lambda)}.
\end{equation}

We use this inequality to prove that 
$(k,a^2cx^k) \in S$, where \[S =
\{(7,d)| d < 1035\cdot 2^7\} \cup \{(8,d)| d < 10\cdot 2^8\}.\]

\begin{lemma}\label{notS}
If $x$, $y$, $z\in \Z^+$
satisfy equation~(\ref{maineq}), then $(k,a^2cx^k) \in S$.
\end{lemma}

\begin{proof}
Suppose that $k \geq 10$.  Then, by inequality~(\ref{Lambdak}),
$\lambda < \Lambda(k)$ and, since $\Lambda$ is a decreasing function,
$\Lambda(k) \leq \Lambda(10)$.
A direct calculation then yields that $\lambda < 3.7$.  
Also, $\mu_k \leq \sqrt{k}$ and
$uv^2 = a^2cx^k - 1 \geq 2^{10} - 1$.
Thus, since $\alpha^k = 1 + \frac{1}{uv^2}$,
inequality~(\ref{uv^2}) implies that
\begin{align*}
63 & < (2^{10} - 1)^{.6}
< \left(uv^2\right)^{(k-2\lambda-2)} 
 <  2^{8}\mu_k^{2} 
\left({\alpha^k}\right)^{2+4\lambda/k}
k^{-(k-2\lambda)} \\
& \leq 2^{8}\left({\alpha^k}\right)^{2+4\lambda/k}
k^{-(k-2\lambda-1)} 
 <  2^8 \left(1 + \frac{1}{2^{10} - 1}\right)^{3.48} 10^{-1.6} 
 <  7,
\end{align*}
a contradiction.  

Next suppose that $k = 9$.  By inequality~(\ref{2^k}) and
a calculation, $\lambda \leq \Lambda_9(2^9) < 3.2$.
Hence,
\begin{align*}
42 & < (2^{9} - 1)^{.6}
< \left(uv^2\right)^{(k-2\lambda-2)}
 <  2^{8}\mu_k^{2} \left({\alpha^k}\right)^{2+4\lambda/k}
k^{-(k-2\lambda)} \\
& < 2^8 \cdot 3 \left(1 + \frac{1}{2^{9} - 1}\right)^{3.43}
9^{-2.6} < 3,
\end{align*}
also a contradiction.

Now, if $k = 8$ and $(k,a^2cx^k) \notin S$, then,
$a^2cx^k \geq 10\cdot 2^8$.
Since $\Lambda_8$ is a decreasing function, a direct
calculation yields
$\lambda \leq \Lambda_8(10\cdot 2^8) < 2.86$.
As above, we find that
\begin{align*}
9  & < (10\cdot 2^{8} - 1)^{.28}
< \left(uv^2\right)^{(k-2\lambda-2)}
 <  2^{8}\mu_k^{2} \left({\alpha^k}\right)^{2+4\lambda/k}
k^{-(k-2\lambda)} \\
& < 2^8 \cdot 2^2 \left(1 + \frac{1}{10\cdot 2^{8} - 1}\right)^{3.43}
8^{-2.28} < 9,
\end{align*}
again, a contradiction.

Finally, if $k = 7$ and $(k,a^2cx^k) \notin S$, then we have
$a^2cx^k \geq 1035\cdot 2^7$ and so
$\lambda \leq \Lambda_7(1035\cdot 2^7) < 2.4162$.
It follows that
\begin{align*}
7.218  & < (1035\cdot 2^7 - 1)^{.1676}
< \left(uv^2\right)^{(k-2\lambda-2)}
 <  2^{8}\mu_k^{2} \left({\alpha^k}\right)^{2+4\lambda/k}k^{-(k-2\lambda)} \\
& < 2^8 7^{1/3} \left(1 + \frac{1}{1035\cdot 2^7 - 1}\right)^{3.3807} 
7^{-2.1676} < 7.213,
\end{align*}
completing the proof.
\end{proof}

Now that we have the values of
$(k,a^2cx^k)$ restricted to a finite set, we employ
continued fractions, in particular, that of $\alpha/x$.
For ease of notation,
let \[C = C_{k,a,c} = \left(\frac{2^kac - 2}{2^kac}\right)^{1/k}.\]

Returning to equation~(\ref{maineq}), we have
$(a^2cx^k - 1)b^2cy^k - a^2cx^k = (abcz^{k})^2 - 2abcz^k$,
and so $(uv^2)b^2cy^k > a^2b^2c^2z^{k}(z^k - 2/(abc))
\geq a^2b^2c^2z^{k}(z^k - 2/(ac))$.
Multiplying each side by $x^k/(uv^2b^2cz^{2k})$ yields
\[\beta^k = \frac{x^ky^k}{z^{2k}} >
\frac{a^2cx^k}{uv^2}\cdot\frac{z^k - 2/(ac)}{z^k}
\geq C^k\alpha^k.\]
Thus, $\alpha > \beta > C\alpha$ and so,
for each $0\leq i \leq k - 1$, 
$\alpha^{k-1-i}\beta^i >
C^{k-1}\alpha^{k-1}$.
Combining this with inequality~(\ref{alpk-betk}), we have
\[\frac{2\alpha^k}{uvw + 1} > \alpha^k - \beta^k >
(\alpha - \beta)k C^{k-1}\alpha^{k-1}\]
and, solving for $\alpha - \beta$, we obtain
\begin{equation}\label{alpha-beta2}
\alpha - \beta  < 
\frac{2\alpha}{k(uvw + 1)} C^{-k+1}
 \leq \frac{2\alpha}{k a c z^k} C^{-k+1}.
\end{equation}
Thus 
\[\left|\frac{\alpha}{x} - \frac{y}{z^2}\right| = \frac{\alpha - \beta}{x}
< \frac{2\alpha}{k a c x z^k} C^{-k+1} < \frac{1}{2z^4}\]
and, therefore, $y/z^2$ is a continued fraction
convergent of $\alpha/x$.

Using standard notation, 
let $[a_0,a_1,a_2,\dots]$ be the continued fraction
expansion of $\alpha/x$, 
with corresponding convergents denoted by $p_i/q_i$.
Fix $J \geq 0$ such that $y/z^2 = p_J/q_J$ and note that,
since $\alpha > \beta$, $J$ is even.  Further, from the definition
of $\alpha$, we have that $\alpha < 2$ and so $\alpha/x < 1$.  Hence,
$p_0/q_0 = a_0 = 0$.  Since $y > 0$, $J \neq 0$.

Combining inequalities~(\ref{alpha-beta>}) and~(\ref{alpha-beta2}),
\[(8k\mu_kuv^2)^{-1}z^{-2\lambda} < 
\frac{2\alpha}{k a c z^k} C^{-k+1}\]
and thus
\[z^{k-2\lambda} < 16\mu_k \alpha \frac{uv^2}{a c} C^{-k+1}.\]
Since $\gcd(p_J,q_J) = 1$, this implies that
\begin{equation}\label{qJ}
q_J \leq z^2 < 
\left(16\mu_k \alpha \frac{uv^2}{a c}C^{-k+1}\right)^{2/(k-2\lambda)}.
\end{equation}
For each possible $(k,a,c,x)$ such that $(k,a^2cx^k)\in S$, 
we compute this upper bound for $q_J$.  Since
the $q_i$ form an increasing sequence, a direct computation of the
convergents of $\alpha/x$ yields a finite number of possible
values for $J$. 

Now, from inequality~(\ref{wlb}),
$(uvw)^2 > k^{k}u^{k}v^{2k}$ and so
$abcz^k > uvw > k^{k/2}u^{k/2}v^{k}$.
Reducing equation~(\ref{maineq}) modulo $b$ yields
$a^2cx^k \equiv 0 \pmod b$ and hence $b\leq a^2cx^k$.
Thus, $z^k > (\sqrt{kuv^2})^ka^{-3}c^{-2}x^{-k}$ implying that
\begin{equation}\label{z>}
z > \sqrt{kuv^2}a^{-3/k}c^{-2/k}x^{-1}.
\end{equation}

Using standard properties of continued fractions,
\[\frac{\alpha}{x} - \frac{y}{z^2} > \frac{1}{q_J^2(a_{J+1} + 2)}
\geq \frac{1}{z^4(a_{J+1} + 2)}.\]
Again, using inequality~(\ref{alpha-beta2}), we have that
\[\frac{1}{z^4(a_{J+1} + 2)} <  
\frac{2\alpha}{k a c x z^k} C^{-k+1}\]
and so, using inequality~(\ref{z>}),
\begin{equation}\label{aJ+1}
a_{J+1} > \frac{kacxz^{k-4}}{2\alpha}C^{k-1} - 2
> \frac{kacx}{2\alpha}\left(\frac{\sqrt{kuv^2}}{a^{3/k}c^{2/k}x}\right)^{k-4}
C^{k-1} - 2.
\end{equation}
But a direct calculation of $a_{J+1}$ for each possible $(k,a,c,x,J)$
yields a contradiction to inequality~(\ref{aJ+1}).

Thus, there is no such solution to equation~(\ref{maineq}).

\vskip 10pt
\noindent
Department of Mathematics \\
Bryn Mawr College \\
Bryn Mawr, PA 19010 \\
egoedhart@brynmawr.edu\\
grundman@brynmawr.edu

\end{document}